\documentclass[10pt]{amsart}
\usepackage{amssymb, amsthm, amsmath}
\usepackage[all]{xy}
\usepackage{graphicx}
\usepackage{psfrag}

\newtheorem{prop}{Proposition}
\newtheorem{thm}{Theorem}
\newtheorem{cor}{Corollary}
\newtheorem{lemma}{Lemma}
\theoremstyle{definition}

\newtheorem*{remark}{Remark}

\newcommand\A{{\mathbb A}}
\newcommand\C{{\mathbb C}}
\newcommand\Q{{\mathbb Q}}

\newcommand\IH{{\mathbb H}}

\newcommand\cA{{\mathcal A}}
\newcommand\cC{{\mathcal C}}

\newcommand\Z{{\mathbb Z}}
\newcommand\cP{{\mathcal P}}
\newcommand\cQ{{\mathcal Q}}

\newcommand\al{\alpha}
\newcommand\la{\lambda}
\newcommand\s{{\sigma}}
\newcommand\ta{{\tau}}

\newcommand\ssm{\smallsetminus}
\newcommand\gequ{\geq}
\newcommand\lequ{\leq}
\newcommand\noin{\noindent}

\newcommand\bull{{\scriptscriptstyle \bullet}}
\newcommand\eqto{\stackrel{\lower1.5pt\hbox{$\scriptstyle\sim\,$}}\to}
\newcommand\ov{\overline}

\newcommand\wt{\widetilde}
\DeclareMathOperator{\Pf}{Pfaffian}
\DeclareMathOperator{\IG}{IG}
\DeclareMathOperator{\OG}{OG}
\DeclareMathOperator{\HH}{\mathrm{H}}
\DeclareMathOperator{\QH}{\mathrm{QH}}
\newcommand{\pic}[2]{\includegraphics[scale=#1]{#2}}

\psfrag{k}{$k$}
\psfrag{(r,c)}{$(r,c)$}
\psfrag{(r',c')}{$(r',c')$}

\begin{document}

\title[Quantum Giambelli formulas for isotropic Grassmannians]
{Quantum Giambelli formulas for isotropic Grassmannians}

\date{December 4, 2008}

\author{Anders Skovsted Buch}
\address{Department of Mathematics, Rutgers University, 110
  Frelinghuysen Road, Piscataway, NJ 08854, USA}
\email{asbuch@math.rutgers.edu}

\author{Andrew Kresch}
\address{Institut f\"ur Mathematik,
Universit\"at Z\"urich, Winterthurerstrasse 190,
CH-8057 Z\"urich, Switzerland}
\email{andrew.kresch@math.uzh.ch}

\author{Harry~Tamvakis} \address{University of Maryland, Department of
Mathematics, 1301 Mathematics Building, College Park, MD 20742, USA}
\email{harryt@math.umd.edu}


\thanks{The authors were supported in part by NSF Grant DMS-0603822
  (Buch), the Swiss National Science Foundation (Kresch), and NSF
  Grant DMS-0639033 (Tamvakis).}

\begin{abstract}
Let $X$ be a symplectic or odd orthogonal Grassmannian which
parametrizes isotropic subspaces in a vector space equipped with a
nondegenerate (skew) symmetric form. We prove quantum Giambelli
formulas which express an arbitrary Schubert class in the small
quantum cohomology ring of $X$ as a polynomial in certain special
Schubert classes, extending the cohomological Giambelli formulas of
\cite{BKT2}.
\end{abstract}

\maketitle 

\setcounter{section}{-1}

\section{Introduction}

Let $E$ be an even (respectively, odd) dimensional complex vector
space equipped with a nondegenerate skew-symmetric (respectively,
symmetric) bilinear form. Let $X$ denote the Grassmannian which
parametrizes the isotropic subspaces of $E$. The cohomology ring
$\HH^*(X,\Z)$ is generated by certain special Schubert classes, which
for us are (up to a factor of two) the Chern classes of the universal
quotient vector bundle over $X$. These special classes also generate
the small quantum cohomology ring $\QH(X)$, a $q$-deformation of
$\HH^*(X,\Z)$ whose structure constants are given by the three point,
genus zero Gromov-Witten invariants of $X$.  In \cite{BKT2}, we proved
a Giambelli formula in $\HH^*(X,\Z)$, that is, a formula expressing a
general Schubert class as an explicit polynomial in the special
classes.  Our goal in the present work is to extend this result to a
formula that holds in $\QH(X)$.

The quantum Giambelli formula for the usual type A Grassmannian was
obtained by Bertram \cite{Be}, and is in fact identical to the
classical Giambelli formula. In the case of maximal isotropic
Grassmannians, the corresponding questions were answered in
\cite{KT1,KT2}. The main conclusions here are similar to those of
loc.\ cit., provided that one uses the raising operator Giambelli
formulas of \cite{BKT2} as the classical starting point. For an odd
orthogonal Grassmannian, we prove that the quantum Giambelli formula
is the same as the classical one. The result is more interesting when
$X$ is the Grassmannian $\IG(n-k,2n)$ parametrizing
$(n-k)$-dimensional isotropic subspaces of a symplectic vector space
$E$ of dimension $2n$. Our theorem in this case states that the
quantum Giambelli formula for $\IG(n-k,2n)$ coincides with the
classical Giambelli formula for $\IG(n+1-k,2n+2)$, provided that the
special Schubert class $\s_{n+k+1}$ is replaced with $q/2$. In a
sequel to this paper, we will discuss the classical and quantum
Giambelli formulas for even orthogonal Grassmannians.

\section{Preliminary Results}
\label{prelims}

\subsection{}
Choose $k\geq 0$ and consider the Grassmannian $\IG=\IG(n-k,2n)$ of
isotropic $(n-k)$-dimensional subspaces of $\C^{2n}$, equipped with a
symplectic form. A partition $\la=(\la_1\geq\ldots\geq\la_\ell)$ is
{\em $k$-strict} if all of its parts greater than $k$ are distinct
integers. Following \cite{BKT1}, the Schubert classes on $\IG$ are
parametrized by the $k$-strict partitions whose diagrams fit in an
$(n-k)\times (n+k)$ rectangle; we denote the set of all such
partitions by $\cP(k,n)$. Given any partition $\la\in\cP(k,n)$ and a
complete flag of subspaces
\[
F_\bull: \, 0=F_0 \subsetneq F_1 \subsetneq \cdots
\subsetneq F_{2n}=\C^{2n}
\]
such that $F_{n+i}= F_{n-i}^\perp$ for
$0\leq i \leq n$, we have a Schubert variety
\[ X_\lambda(F_\bull) := \{ \Sigma \in \IG \mid \dim(\Sigma \cap
   F_{p_j(\lambda)}) \gequ j \ \ \forall\, 1 \lequ j \lequ
   \ell(\lambda) \} \,,
\]
where $\ell(\la)$ denotes the number of (non-zero) parts of $\la$ and 
\[
p_j(\lambda) := n+k+j-\lambda_j - \#\{i<j : \lambda_i+\lambda_j
> 2k+j-i \}.
\]
This variety has codimension $|\la|=\sum \la_i$ and defines, via
Poincar\'e duality, a Schubert class $\s_{\la}=[X_{\la}(F_\bull)]$ in
$\HH^{2|\la|}(\IG,\Z)$. The Schubert classes $\s_\la$ for $\la\in
\cP(k,n)$ form a free $\Z$-basis for the cohomology ring of $\IG$.
The {\em special Schubert classes} are defined by
$\s_r=[X_r(F_\bull)]=c_r(\cQ)$ for $1\lequ r \lequ n+k$, where $\cQ$
denotes the universal quotient bundle over $\IG$.

The classical Giambelli formula for $\IG$ is expressed using 
Young's {\em raising operators} \cite[p.\ 199]{Y}. We first agree that
$\s_0=1$ and $\s_r=0$ for $r<0$. For any integer sequence
$\alpha=(\alpha_1,\alpha_2,\ldots)$ with finite support and $i<j$, we
set $R_{ij}(\alpha) =
(\alpha_1,\ldots,\alpha_i+1,\ldots,\alpha_j-1, \ldots)$; a raising
operator $R$ is any monomial in these $R_{ij}$'s.  Define $m_\alpha =
\prod_i \s_{\alpha_i}$ and $R\,m_{\al} = m_{R\al}$ for any raising
operator $R$. For any $k$-strict partition $\la$, we consider the 
operator  
\[
R^{\la} =\prod (1-R_{ij})\prod_{\la_i+\la_j >
 2k+j-i}(1+R_{ij})^{-1}
\]
where the first product is
over all pairs $i<j$ and second product is over pairs $i<j$ such that
$\la_i+\la_j > 2k+j-i$.
The main result of \cite{BKT2} states that the {\em Giambelli formula}
\begin{equation}
\label{clasgiam}
\s_\la =R^{\la}\,m_{\la}
\end{equation}
holds in the cohomology ring of $\IG(n-k,2n)$.

\subsection{}
\label{pieriC}
As is customary, we will represent a partition by its Young diagram of
boxes; this is used to define the containment relation for partitions.
Given two diagrams $\mu$ and $\nu$ with $\mu\subset\nu$, the
skew diagram $\nu/\mu$ (i.e., the set-theoretic difference $\nu\ssm\mu$)
is called a horizontal (resp.\ vertical) strip if it does not contain 
two boxes in the same column (resp.\ row). 

We say that the box $[r,c]$ in row $r$ and column $c$ of a $k$-strict
partition $\lambda$ is {\em $k$-related\/} to the box $[r',c']$ if
$|c-k-1|+r = |c'-k-1|+r'$. For instance, the grey boxes in the
following partition are $k$-related.
\[ \pic{0.65}{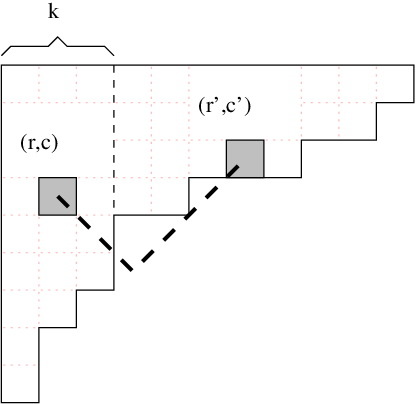} \]
For any two $k$-strict partitions $\lambda$ and $\mu$, we write
$\lambda \to \mu$ if $\mu$ may be obtained by removing a vertical
strip from the first $k$ columns of $\lambda$ and adding a horizontal
strip to the result, so that

\medskip
\noin (1) if one of the first $k$ columns of $\mu$ has the same number
of boxes as the same column of $\lambda$, then the bottom box of this
column is $k$-related to at most one box of $\mu \smallsetminus
\lambda$; and

\medskip
\noin
(2) if a column of $\mu$ has fewer boxes than the same column of
$\lambda$, then the removed boxes and the bottom box of $\mu$ in this
column must each be $k$-related to exactly one box of $\mu
\smallsetminus \lambda$, and these boxes of $\mu \smallsetminus
\lambda$ must all lie in the same row.

\medskip

Let $\A$ denote the set of boxes of $\mu\ssm \la$ in columns $k+1$
through $k+n$ which are not mentioned in (1) or (2) above, and define
$N(\lambda,\mu)$ to be the number of connected components of $\A$
which do not have a box in column $k+1$.  Here two boxes are connected
if they share at least a vertex. In \cite[Theorem 1.1]{BKT1} we proved
that the Pieri rule
\begin{equation}
\label{classpieri}
\sigma_p \cdot \sigma_\lambda = \sum_{\substack{\lambda \to \mu \\
|\mu|=|\lambda|+p}} 2^{N(\lambda,\mu)} \, \sigma_\mu \,
\end{equation}
holds in $\HH^*(\IG,\Z)$, for any $p\in [1,n+k]$.

\subsection{} 
In the following sections we will work in the stable cohomology ring
$\IH(\IG_k)$, which is the inverse limit in the category of graded
rings of the system
\[
\cdots \leftarrow \HH^*(\IG(n-k,2n),\Z) \leftarrow
\HH^*(\IG(n+1-k,2n+2),\Z) \leftarrow \cdots
\]
The ring $\IH(\IG_k)$ has a free $\Z$-basis of Schubert classes
$\s_\la$, one for each $k$-strict partition $\la$, and may be 
presented as a quotient of the polynomial ring 
$\Z[\s_1,\s_2,\ldots]$ modulo the relations
\begin{equation}
\label{stabrels}
\s_r^2+ 2\sum_{i=1}^r(-1)^i\s_{r+i}\s_{r-i} = 0 \ \ \ \text{for $r>k$}.
\end{equation}
There is a natural surjective ring homomorphism
$\IH(\IG_k)\to\HH(\IG(n-k,2n),\Z)$ that maps $\s_\la$ to $\s_\la$,
when $\la\in\cP(k,n)$, and to zero, otherwise. The Giambelli 
formula (\ref{clasgiam}) and Pieri rule (\ref{classpieri}) are both 
valid in $\IH(\IG_k)$. We begin with some elementary consequences 
of these theorems.

For any $k$-strict partition $\la$ of length $\ell$, we define the 
sets of pairs
\[
\cA(\la) = \{(i,j)\ |\ \la_i+\la_j \leq 2k+j-i
\ \, \text{and} \ \, 1\leq i< j \leq \ell\}
\]
\[
\cC(\la) = \{(i,j)\ |\ \la_i+\la_j > 2k+j-i
\ \, \text{and} \ \, 1\leq i< j \leq \ell\}
\]
and two integer vectors $a=(a_1,\ldots,a_{\ell})$ and 
$c= (c_1,\ldots,c_{\ell})$ by setting
\[
a_i = \#\{j\ |\ (i,j)\in\cA(\la)\}, \  \
c_i = \#\{j\ |\ (i,j)\in\cC(\la)\}
\]
for each $i$.

\begin{prop}
\label{ac}
We have $\la_i-c_i\geq \la_j-c_j$ for each $i<j\leq\ell$.
\end{prop}
\begin{proof}
Observe that the desired inequality is equivalent to
\begin{equation}
\label{diff}
\la_i-\la_j \geq
\#\{r\leq \ell\ |\ (i,r)\in \cC(\la)\} - 
\#\{r\leq \ell\ |\ (j,r)\in \cC(\la)\}.
\end{equation}
Let $j=i+r$ and let $s$ (respectively $t$) be maximal such that
$(i,s)\in\cC(\la)$ (respectively, $(j,t)\in\cC(\la)$). Assume first
that $t$ exists, hence $s$ exists and $s\gequ t$. The inequality
(\ref{diff}) then becomes $\la_i-\la_{i+r} \gequ s-t+r$. We have
\[
\la_i+\la_s\geq 2k+1+s-i \ \text{ and } \
\la_{i+r}+\la_{t+1}\leq 2k+t+1-i-r,
\]
hence
\[
\la_i-\la_{i+r}\geq 
s-t+r+(\la_{t+1}-\la_s).
\]
If $t<s$, then $\la_{t+1}\geq \la_s$ and we are done. If $t=s$, we
need to show that $\la_i-\la_{i+r} \geq r$. This is true because
$(j,j+1)\in\cC(\la)$ and $\la$ is $k$-strict, 
hence $\la_i>\la_{i+1}>\cdots > \la_{i+r}$.

Next we assume that $t$ does not exist, so that either $j=\ell$ or
the pair $(j,j+1)$ lies in $\cA(\la)$ and
\begin{equation}
\label{ina}
\la_j+\la_{j+1}\lequ 2k+1.
\end{equation}
If $s$ does not exist, there is nothing to prove.
We must show that $\la_i-\la_j\gequ s-i$,
knowing that $(i,s)\in\cC(\la)$, that is,
\begin{equation}
\label{inC}
\la_i+\la_s\gequ 2k+1+s-i.
\end{equation}
Assume first that $\la_s\gequ \la_j$. If $\la_s>k$ then we have
\[
\la_i > \la_{i+1} > \cdots > \la_s
\]
and hence $\la_i-\la_j\gequ \la_i-\la_s \gequ s-i$.
Otherwise $\la_s\lequ k$ and
(\ref{inC}) gives
\[
\la_i-\la_j\gequ \la_i-\la_s \gequ \la_i -k \gequ s-i+1+(k-\la_s)\gequ s-i.
\]
Finally, suppose that $\la_s<\la_j$, so in particular $j+1\lequ s$. Then
(\ref{ina}) and (\ref{inC}) give
\begin{gather*}
\la_i-\la_j \gequ \la_i+(\la_{j+1}-2k-1)
\gequ (2k+1+s-i-\la_s)+\la_{j+1}-2k-1 \\
=(\la_{j+1}-\la_s)+(s-i)\gequ s-i. \qedhere
\end{gather*} 
\end{proof}
\noindent
Proposition \ref{ac} implies that for any $\lambda$, the composition
$\lambda-c$ is a partition, while $\lambda+a$ is a strict partition.

\begin{prop}
\label{giamgen}
For any $k$-strict partition $\la$, the Giambelli polynomial 
$R^\la \,m_\la$ for $\s_\la$ involves only generators 
$\s_p$ with $p\leq \la_1+a_1+\la_2+a_2$. 
\end{prop}
\begin{proof}
We have
\[
R^\la\,m_\la = \prod_{1\leq i<j\leq\ell}\frac{1-R_{ij}}{1+R_{ij}}\,
\prod_{(i,j)\in\cA(\la)}(1+R_{ij})\, m_\la = 
\sum_{\nu\in N} \prod_{1\leq i<j\leq\ell}\frac{1-R_{ij}}{1+R_{ij}}
\, m_\nu
\]
where $N$ is the multiset of integer vectors defined by
\[
N = \left\{\prod_{(i,j)\in S}R_{ij}\,\la \ \
|\ \ S\subset \cA(\la)\right\}.
\]
If $m>0$ is the least integer such that $2m \geq \ell$, then 
we have 
\begin{equation}
\label{pfaff}
\prod_{1\leq i<j\leq m}\frac{1-R_{ij}}{1+R_{ij}} = 
\Pf\left(\frac{1-R_{ij}}{1+R_{ij}}\right)_{1\leq i,j \leq 2m}\,.
\end{equation}
Equation (\ref{pfaff}) follows from Schur's classical identity
\cite[Sec.\ IX]{S2}
\[
\prod_{1\leq i<j\leq 2m}\frac{x_i-x_j}{x_i+x_j} =
\Pf\left(\frac{x_i-x_j}{x_i+x_j}\right)_{1\leq i,j \leq 2m}.
\]
Note that each single entry in the Pfaffian (\ref{pfaff}) expands 
according to the formula
\[
\frac{1-R_{12}}{1+R_{12}}\, m_{c,d} = 
\s_c\,\s_d - 2\,\s_{c+1}\,\s_{d-1} + 2\, \s_{c+2}\,\s_{d-2} - \cdots
+(-1)^d\, 2\, \s_{c+d}.
\]
By Proposition \ref{ac}, we know that $\la+a =
(\la_1+a_1,\la_2+a_2,\ldots)$ is a strict partition, hence
$\la_i+a_i+\la_j+a_j \leq \la_1+a_1+\la_2+a_2$ for any distinct $i$
and $j$.  Since we furthermore have $\nu_i\leq \la_i+a_i$, for any
$\nu\in N$, the result follows.
\end{proof}

\begin{cor}
\label{onlycor}
For any $\la\in \cP(k,n)$ the stable Giambelli polynomial for $\s_\la$
involves only special classes $\s_p$ with $p\leq 2n+2k-1$.
\end{cor}

\begin{lemma}
\label{stablepieri}
Let $\la$ and $\nu$ be $k$-strict partitions such that $\nu_1 >
\max(\la_1,\ell(\la)+2k)$ and $p\geq 0$. Then the coefficient of
$\s_\nu$ in the Pieri product $\s_p \cdot \s_\la$ is equal to the coefficient
of $\s_{(\nu_1+1,\nu_2,\nu_3,\ldots)}$ in the product $\s_{p+1} \cdot \s_\la$.
\end{lemma}
\begin{proof}
Let $c=\max(\la_1,\ell(\la)+2k)+1$. Observe that box $[1,c]$ belongs
to a connected component of the subset $\A$ of $\nu\ssm\la$ defined in
\S \ref{pieriC} which extends all the way to the rightmost box
of $\nu$. The same statement is true for $(\nu_1+1,\nu_2,
\nu_3,\ldots) \ssm \la$, except that the component goes one box
further to the right. The number of components of $\A$ which do not
meet column $k+1$ in both cases is the same, hence the two Pieri
coefficients are equal.
\end{proof}

Given any partition $\la$, we let $\la^*=(\la_2,\la_3,\ldots)$.

\begin{prop}
\label{recprop}
For any $\la\in \cP(k,n)$, there exists a recursion formula of the form
\begin{equation}
\label{recurseC}
\s_\la = \sum_{p=\la_1}^{2n+2k-1} \
\sum_{\mu \subset \la^*} a_{p,\mu} \,\s_p \,\s_\mu
\end{equation}
with $a_{p,\mu}\in\Z$, valid in the stable cohomology ring $\IH(\IG_k)$
\end{prop}
\begin{proof}
The argument is done in two steps, the first one being a reduction
step.  We claim that it is enough to prove that there exists a
nonnegative integer $m$ such that $\s_{(\la_1+m,\la^*)}$ is a linear
combination of $\s_p \,\s_\mu$ for $\la_1+m\leq p\leq 2n+2k-1+m$ and
$\mu \subset \la^*$. Suppose that we know this, then let us try to
obtain an expression for $\s_\la$.  

If $\la_1\geq \ell(\la)+2k-1$, and if we have an expression
\begin{equation}
\label{reduce}
\s_{(\la_1+m,\la^*)}=
\sum_{p=\la_1+m}^{2n+2k-1+m} \
\sum_{\mu \subset \la^*} a_{p,\mu} \, \s_p \, \s_\mu
\end{equation}
then we must have 
\begin{equation}
\label{secsum}
\s_\la=\sum_{p=\la_1}^{2n+2k-1}\sum_{\mu} a_{p+m,\mu} \,\s_p \,\s_\mu. 
\end{equation}
Indeed, upon applying the Pieri rule (\ref{classpieri}), the
coefficient of $\s_\nu$ for $\nu$ with $\nu_1>\la_1$ in each term in
the sum (\ref{secsum}) is equal to the coefficient of
$\s_{(\nu_1+m,\nu_2,\ldots)}$ in the corresponding term in (\ref{reduce})
by Lemma \ref{stablepieri}, and by (\ref{reduce}) these sum to
zero. It remains to consider $\nu_1=\la_1$, i.e., $\nu=\la$, and the
coefficient in this case is 1 since we must have
$a_{\la_1+m,\la^*}=1$.

If $\la_1<\ell(\la)+2k-1$, then set $\la'=(n+k,\la^*)$. By the above
case, we have a recursion
\[
\s_{\la'} = \sum_{p=n+k}^{2n+2k-1}\ 
\sum_{\mu\subset\la^*} a_{p,\mu} \,\s_p \,\s_\mu
\]
for some $a_{p,\mu}\in\Z$. Using Lemma \ref{stablepieri}, now, we
deduce that
\[
\s_\la= \sum_{p=\la_1}^{n+k+\la_1-1} \
\sum_{\mu\subset\la^*}  a_{p+n+k-\la_1,\mu} \,\s_p \,\s_\mu
+\sum_\nu b_{\la\nu}\, \s_\nu
\]
where $b_{\la\nu}\in \Z$ and the partitions $\nu$ in the second sum
satisfy $\la_1 < \nu_1 \leq \ell(\la)+2k-1$ and
$\nu^*\subset\la^*$. By decreasing induction on $\nu_1$, we may assume
that expressions for these $\s_\nu$ as linear combinations of $\s_p \,
\s_\mu$ with $\nu_1\leq p\leq2n+2k-1$ and $\mu \subset \nu^*$
exist. This completes the proof of the claim.

In the second step, given $\la \in \cP(k,n)$ and $m>|\la|$, we show that
$\s_{(\la_1+m,\la^*)}$ is a linear combination of products $\s_p\,
\s_\mu$ for $\la_1+m\leq p\leq 2n+2k-1+m$ and $\mu \subset \la^*$.
This uses the following result.

\begin{lemma}
\label{prodlemma}
Let $P_r$ be the set of partitions $\mu$ with $|\mu| = r$, and
let $m$ be a positive integer. Then the $\Z$-linear map
\[
\phi\ :\ \bigoplus_{r=0}^{\lfloor\frac{m-1}{2}\rfloor}
\bigoplus_{\mu\in P_r}\Z \to \IH(\IG_k)
\]
which, for given $r$ and $\mu\in P_r$, sends the corresponding
basis element to $\s_{m-r} \s_\mu$, is injective.
\end{lemma}
\begin{proof}
The image of $\phi$ is contained in the span of the $\s_{(m-r,\mu)}$
for $0\leq r <\frac{m}{2}$ and $\mu$ in $P_r$. Observe that the linear
map $\phi$ is represented by a block triangular matrix with diagonal
matrices as the blocks along the diagonal.  The lemma follows.
\end{proof}

There are two elementary ways to obtain a recursion formula for a
given Schubert class. First, for any $k$-strict partition $\la$, the 
Pieri rule (\ref{classpieri}) gives
\begin{equation}
\label{eq1}
\s_\la = \s_{\la_1} \s_{\la^*} - \sum_{\substack{\mu_1>\la_1 \\
\mu^*\subset\la^*}} d_{\la\mu}\,\s_\mu,
\end{equation}
where the $d_{\la\mu}\in\Z$ and the sum is over partitions $\mu$ with
$\mu_1>\la_1$ and $\mu^*\subset\la^*$. We then apply the same
prescription to each of the summands $\s_\mu$ in (\ref{eq1}), and
iterate this procedure. Finally, we obtain an expression
\[
\s_\la = \sum_{p=\la_1}^{|\la|}\ \sum_{\mu\subset\la^*}
a_{p,\mu} \,\s_p\,\s_\mu.
\]
Second, consider the stable Giambelli formula
\begin{equation}
\label{helpful}
\s_{\la} = R^\la\, m_\la = \sum_\nu b_\nu\, m_\nu
\end{equation}
in the ring $\IH(\IG_k)$. By Proposition \ref{giamgen} we know that 
the integer vectors $\nu$ in (\ref{helpful}) all
satisfy $\nu_1\leq \la_1+a_1+\la_2+a_2$. Hence we have an equation
\[
\s_\la = \sum_{p=\la_1}^{\la_1+a_1+\la_2+a_2} 
\s_p \ \sum_{\nu \, :\, \nu_1=p} b_\nu \, m_{\nu^*}.
\]

For $\la\in \cP(k,n)$, choose $m > |\la|$, and set
$\la'=(\la_1+m,\la^*)$.  Consider the expressions obtained by the two
methods described in the last paragraph applied to $\la'$:
\[
\s_{\la'}= \sum_{p=\la_1+m}^{|\la|+m} \sum_{\mu \subset \la^*} \
a_{p,\mu}\, \s_p \,\s_\mu
\]
and 
\[
\s_{\la'}= \sum_{p=\la_1+m}^{2n+2k-1+m}\ 
\sum_{\mu \in P_{|\la|+m-p}} b_{p,\mu} \, \s_p \, \s_\mu.
\]
By Lemma \ref{prodlemma}, we have $a_{p,\mu}=b_{p,\mu}$.  Hence, in
particular, $a_{p,\mu}=0$ whenever $p>2n+2k-1+m$. Therefore we have a
recursion formula (\ref{recurseC}) for $\s_{\la'}$, as desired.
\end{proof}

\begin{remark}
One can be more precise about the recursion formula (\ref{recurseC}) in
the case when the $k$-strict partition $\la\in \cP(k,n)$ satisfies 
$\la_1\geq \ell(\la)+2k-1$. If the Pieri rule reads
\[
\s_{\la_1}\cdot \s_{\la^*}  = 
\sum_{p=\la_1}^{2n+2k-1} \ \sum_{\mu\subset\la^*} 2^{n(p,\mu)}\, \s_{p,\mu}
\]
then we have
\[
\s_{\la} = \sum_{p=\la_1}^{2n+2k-1} \ \sum_{\mu\subset\la^*} 
(-1)^{p-\la_1} \, 2^{n(p,\mu)}\, \s_p\, \s_\mu.
\]
This result is proved in \cite{T}.
\end{remark}

\section{Quantum Giambelli for $\IG(n-k,2n)$}
\label{qgig}

The quantum cohomology ring $\QH^*(\IG)$ is a $\Z[q]$-algebra
which is isomorphic to $\HH^*(\IG,\Z)\otimes_{\Z}\Z[q]$ as a module over
$\Z[q]$. The degree of the formal variable $q$ here is $n+k+1$.
We begin by recalling the quantum Pieri rule of \cite{BKT1}. This 
states that for any $k$-strict partition $\lambda\in\cP(k,n)$
and integer $p \in [1,n+k]$, we have
\begin{equation}
\label{igqpieri}
  \s_p \cdot \s_\lambda = \sum_{\lambda \to \mu} 2^{N(\lambda,\mu)}\,\s_\mu +
  \sum_{\lambda \to \nu} 2^{N(\lambda,\nu)-1} \,\s_{\nu^*}\, q
\end{equation}
in the quantum cohomology ring of $\IG(n-k,2n)$.  The first sum in
(\ref{igqpieri}) is over partitions $\mu\in\cP(k,n)$ such that
$|\mu|=|\lambda|+p$, and the second sum is over partitions
$\nu\in\cP(k,n+1)$ with $|\nu|=|\lambda|+p$ and $\nu_1 = n+k+1$.

We work now with rational coefficients and introduce an
important tool: a ring homomorphism 
\[
\pi:\IH(\IG_k)\to\QH(\IG(n-k,2n)).
\]
The map $\pi$ is determined by setting
\[
\pi(\s_i) = 
\begin{cases}
\s_i & \text{if $1\leq i \leq n+k$}, \\
q/2 & \text{if $i=n+k+1$}, \\
0 & \text{if $n+k+1< i \leq 2n+2k$}, \\
0 & \text{if $i$ is odd and $i>2n+2k$}.
\end{cases}
\]
The relations (\ref{stabrels}) then uniquely specify the values
$\pi(\s_i)$ for $i$ even and $i>2n+2k$.

\begin{thm}[Quantum Giambelli for $\IG$]
\label{qgiamIG}
For every $\la\in\cP(k,n)$, the quantum Giambelli formula for $\s_\la$
in $\QH(\IG(n-k,2n))$ is obtained from the classical Giambelli 
formula $\s_\la = R^\la\,m_\la$ in $\HH^*(\IG(n+1-k,2n+2),\Z)$
by replacing the special Schubert class $\s_{n+k+1}$ with 
$q/2$.
\end{thm}
\begin{proof}
We claim that the ring homomorphism $\pi$ satisfies
$\pi(\s_\la) = \s_\la$ for all $\la\in\cP(k,n)$. The proof of the 
claim is by induction on the length of $\la$, with the case of 
length one being clear.
For the inductive step, Proposition \ref{recprop} implies that  
\begin{equation}
\label{eqrec}
\s_\la = \sum_{p=\la_1}^{n+k+1} \
\sum_{\mu \subset \la^*} a_{p,\mu} \,\s_p \,\s_\mu
\end{equation}
holds in the cohomology ring of $\IG(n+1-k,2n+2)$. Furthermore, if we
we apply the ring homomorphism $\pi$ to both sides of (\ref{recurseC})
and use the induction hypothesis, we find that
\begin{equation}
\label{eqrec2}
\pi(\s_\la) = 
\sum_{p=\la_1}^{n+k} \ \sum_{\mu \subset \la^*} 
a_{p,\mu}\, \s_p \, \s_\mu
+ \frac{q}{2} \sum_{\mu \subset \la^*} a_{n+k+1,\mu} \, \s_\mu
\end{equation}
holds in $\QH^*(\IG(n-k,2n))$. The right hand side of (\ref{eqrec2})
can be evaluated using the quantum Pieri formula (\ref{igqpieri}). We
perform this computation using (\ref{eqrec}) and deduce that the
expression evaluates to $\s_\la$, proving the claim.

According to Corollary \ref{onlycor}, the stable Giambelli polynomial 
for $\s_\la$ may be expressed as an equation
\begin{equation}
\label{stabgiam}
\s_\la = f_\la(\s_1,\ldots,\s_{2n+2k-1})
\end{equation}
in $\IH(\IG_k)$, where $f_\la\in \Z[x_1,\ldots,x_{2n+2k-1}]$. 
We now apply the ring homomorphism $\pi$ to (\ref{stabgiam})
to get an identity in
$\QH(\IG(n-k,2n))$. 
The left hand side evaluates to $\s_\la$ by the last claim, while the 
right hand side maps to $f_\la(\s_1,\ldots,\s_{n+k},\frac{q}{2},0,\ldots,0)$.
We deduce that
\[
\s_\la = f_\la(\s_1,\ldots,\s_{n+k},\frac{q}{2},0,\ldots,0)
\]
in $\QH(\IG(n-k,2n))$, which is precisely the quantum Giambelli formula.
\end{proof}

\section{Quantum Giambelli for $\OG(n-k,2n+1)$}
\label{qgog}

\subsection{}
\label{firstOG}
For each $k\geq 0$, let $\OG=\OG(n-k,2n+1)$ denote the odd orthogonal
Grassmannian which parametrizes the $(n-k)$-dimensional isotropic
subspaces in $\C^{2n+1}$, equipped with a non-degenerate symmetric
bilinear form. The Schubert varieties in $\OG$ are indexed by the same
set of $k$-strict partitions $\cP(k,n)$ as for $\IG(n-k,2n)$.  Given
any $\la\in\cP(k,n)$ and a complete flag of subspaces
\[
F_\bull: \, 0=F_0 \subsetneq F_1 \subsetneq \cdots
\subsetneq F_{2n+1}=\C^{2n+1}
\]
such that $F_{n+i}= F_{n+1-i}^\perp$ for $1 \leq i \leq n+1$, we
define the codimension $|\la|$ Schubert variety
\[ 
   X_\lambda(F_\bull) = \{ \Sigma \in \OG \mid \dim(\Sigma \cap
   F_{\ov{p}_j(\lambda)}) \gequ j \ \ \forall\, 1 \lequ j \lequ 
   \ell(\lambda) \} \,,
\]
where 
\[
\ov{p}_j(\lambda) = n+k+1+j-\lambda_j - \#\{i\leq j : \lambda_i+\lambda_j
> 2k+j-i \}.
\]
Let $\ta_{\la} \in \HH^{2|\lambda|}(\OG,\Z)$ denote the cohomology
class dual to the cycle given by $X_\lambda(F_\bull)$.

Let $\ell_k(\la)$ be the number of parts
$\la_i$ which are strictly greater than $k$, and let $\cQ_{\IG}$ and
$\cQ_{\OG}$ denote the universal quotient vector bundles over
$\IG(n-k,2n)$ and $\OG(n-k,2n+1)$, respectively. It is known (see
e.g.\ \cite[\S 3.1]{BS}) that the map which sends $\s_p =
c_p(\cQ_{\IG})$ to $c_p(\cQ_{\OG})$ for all $p$ extends to a ring
isomorphism $\varphi:\HH^*(\IG,\Q) \to \HH^*(\OG,\Q)$ such that 
$\varphi(\s_\la) = 2^{\ell_k(\la)} \ta_{\la}$ for all $\la\in \cP(k,n)$.

We let $c_p= c_p(\cQ_{\OG})$. The {\em special Schubert classes}
on $\OG$ are related to the Chern classes 
$c_p$ by the equations
\[
c_p=
\begin{cases}
\ta_p & \text{if $p\lequ k$},\\
2\ta_p & \text{if $p> k$}.
\end{cases}
\]
For any integer sequence
$\alpha$, set $m_{\alpha} = \prod_ic_{\al_i}$. Then for
every $\la\in\cP(k,n)$, the classical Giambelli formula
\begin{equation}
\label{mainthmB}
\ta_\la = 2^{-\ell_k(\la)}R^{\la}\,m_{\la}
\end{equation}
holds in $\HH^*(\OG,\Z)$. 

\subsection{}

The quantum cohomology ring $\QH^*(\OG(n-k,2n+1))$ is defined
similarly to that of $\IG$, but the degree of $q$ here is $n+k$.
More notation is required to state the quantum Pieri rule for
$\OG$. For each $\la$ and $\mu$ with $\la\to\mu$, we define
$N'(\la,\mu)$ to be equal to the number (respectively, one less than
the number) of connected components of $\A$, if $p\leq k$
(respectively, if $p>k$). Let $\cP'(k,n+1)$ be the set of
$\nu\in\cP(k,n+1)$ for which $\ell(\nu)=n+1-k$, $2k \lequ \nu_1 \lequ
n+k$, and the number of boxes in the second column of $\nu$ is at most
$\nu_1-2k+1$.  For any $\nu\in\cP'(k,n+1)$, we let $\wt\nu \in
\cP(k,n)$ be the partition obtained by removing the first row of $\nu$
as well as $n+k-\nu_1$ boxes from the first column.  That is,
\[
\wt{\nu}=(\nu_2,\nu_3,\ldots,\nu_r), \
\text{where $r=\nu_1-2k+1$.}
\]

According to \cite[Theorem 2.4]{BKT1}, for 
any $k$-strict partition $\lambda\in\cP(k,n)$ and integer $p\in [1,n+k]$, 
the following quantum Pieri rule holds in $\QH^*(\OG(n-k,2n+1))$.
\begin{equation}
\label{ogqp}
\ta_p \cdot \ta_\la =
\sum_{\la\to\mu} 2^{N'(\lambda,\mu)}\,\tau_\mu +\sum_{\la\to\nu}
2^{N'(\lambda,\nu)} \,\tau_{\wt{\nu}}\, q \, +\,
\sum_{\la^*\to\rho}
2^{N'(\lambda^*,\rho)} \,\ta_{\rho^*}\,
q^2\, .
\end{equation}
Here the first sum is classical, the second sum is over
$\nu\in \cP'(k,n+1)$ with $\lambda\to\nu$ and $|\nu|=|\la|+p$, and
the third sum is empty unless $\la_1=n+k$, and over
$\rho\in\cP(k,n)$ such that $\rho_1=n+k$, $\lambda^*\to\rho$, and
$|\rho|=|\la|-n-k+p$.

Let $\delta_p=1$, if $p\leq k$, and $\delta_p=2$, otherwise.
The stable cohomology ring $\IH(\OG_k)$ has a free $\Z$-basis of 
Schubert classes $\ta_\la$ for $k$-strict partitions $\la$, 
and is presented as a quotient of the polynomial ring 
$\Z[\ta_1,\ta_2,\ldots]$ modulo the relations
\begin{equation}
\label{stabrelsOG}
\ta_r^2+ 2\sum_{i=1}^r(-1)^i\delta_{r-i}\ta_{r+i}\ta_{r-i} = 0 \ \ \ 
\text{for $r>k$}.
\end{equation}

We require a ring homomorphism 
\[
\tilde\pi:\IH(\OG_k)\to\QH(\OG(n-k,2n+1))
\]
analogous to the map $\pi$ of \S \ref{qgig}. 
The morphism $\tilde\pi$ is determined by setting
\[
\tilde\pi(\ta_i) = 
\begin{cases}
\ta_i & \text{if $1\leq i \leq n+k$}, \\
0 & \text{if $n+k< i < 2n+2k$}, \\
0 & \text{if $i$ is odd and $i>2n+2k$}.
\end{cases}
\]
The relations (\ref{stabrelsOG}) then uniquely specify the values
$\tilde\pi(\ta_i)$ for $i$ even and $i\geq 2n+2k$.  To verify this, we
just have to check that the relations
\[
\ta_r^2+ 2\sum_{i=1}^{n+k-r}(-1)^i\delta_{r-i}\ta_{r+i}\ta_{r-i} = 0
\]
are true in $\QH^*(\OG(n-k,2n+1))$, for $(n+k)/2\leq r \leq n+k-1$.
But when $k<n-1$ the individual terms in these relations carry no $q$
correction.  Indeed, we are applying the quantum Pieri rule
(\ref{ogqp}) to length $1$ partitions, hence the $q$ term vanishes
(since $1 < n-k$) and the $q^2$ term vanishes (since
$\deg(q^2)=2n+2k$). It remains only to consider the case $k=n-1$,
which uses the quantum Pieri rule for the quadric $\OG(1,2n+1)$. The
computation is then done as in \cite[Theorem 2.5]{BKT1} (which treats
the case $r=n$), and involves computing the coefficient $c$ of
$q\,\ta_{2(r-n)+1}$ in the corresponding expression.  As in loc.\
cit., the result is $c=1-2+2-\cdots \pm2\mp 1$ when $r\leq (3n-2)/2$,
and otherwise $c=2-4+4-\cdots \pm 4\mp 2$; hence $c=0$ in both cases.

\begin{thm}[Quantum Giambelli for $\OG$]
For every $\la\in\cP(k,n)$, we have 
\[
\ta_\la = 2^{-\ell_k(\la)}R^{\la}\,m_{\la}
\]
in the quantum cohomology ring $\QH(\OG(n-k,2n+1))$. In other words,
the quantum Giambelli formula for $\OG$ is the same as the classical
Giambelli formula.
\end{thm}
\begin{proof}
We may use the isomorphism $\varphi$ of \S \ref{firstOG} to 
translate all of the results of \S \ref{prelims} to their images 
in $\HH^*(\OG,\Z)$ and the stable cohomology ring $\IH(\OG_k)$.
The proof of quantum Giambelli for $\OG$ is therefore identical 
to the proof of Theorem \ref{qgiamIG}, using the ring homomorphism 
$\tilde\pi$ in place of $\pi$.
\end{proof}

\end{document}